\documentclass[12pt]{amsart}
\usepackage[T1]{fontenc}
\usepackage[utf8]{inputenc}
\usepackage{commath}
\usepackage[ maxnames=4,maxalphanames=4,style=alphabetic-verb,doi=false,
  url=false,
  isbn=false,
  eprint=false
]{biblatex}
\bibliography{main.bib}
\usepackage[margin=1in]{geometry}
\usepackage{graphicx}
\usepackage{caption}
\usepackage{subcaption}
\usepackage{amsmath,amsfonts,amssymb,amsthm,amsaddr,etoolbox}
\usepackage{latexsym,hyperref}
\usepackage[most]{tcolorbox}
\usepackage{xcolor}
\usepackage{scrextend}
\usepackage{quiver}
\usepackage{enumitem}
\usepackage[normalem]{ulem}
\usepackage{mathtools}
\usepackage[linewidth=1pt]{mdframed}
\usepackage{tikz-cd}
\usepackage{float}

\DeclareFieldFormat{postnote}{#1}
\DeclareFieldFormat{multipostnote}{#1}

\author[R. Easwar]{Rohun Easwar}
\address{Independent Researcher}

\email{rohuneaswar@gmail.com}

\author[S. Johari]{Shaurya Johari}
\address{Department of Computer Science and Engineering\\Indian Institute of Technology, Kanpur\\ Uttar Pradesh 208016, India}

\email{shauryaj23@iitk.ac.in}

\title{Conditions for traceability under a bound on the size of even-distance sets}

\keywords{traceability, block graph}
\subjclass[2020]{05C45, 05C12, 05C40}

\newcommand\restr[2]{{
  \left.\kern-\nulldelimiterspace 
  #1 
  \littletaller 
  \right|_{#2} 
  }}

\theoremstyle{plain}
\newtheorem{defn}{Definition}[section]

\newtheorem{lem}[defn]{Lemma}

\newtheorem*{prop*}{Proposition}
\newtheorem*{thm*}{Theorem}
\newtheorem{thm}[defn]{Theorem}

\newtheorem*{claim*}{Claim}

\theoremstyle{remark}
\newtheorem{rem}[defn]{Remark}
\theoremstyle{remark}

\theoremstyle{remark}

\theoremstyle{remark}

\theoremstyle{remark}

\theoremstyle{remark}

\theoremstyle{remark}

\theoremstyle{remark}

\theoremstyle{remark}

\numberwithin{equation}{section}

\theoremstyle{definition}
\newtheorem{introthm}{Theorem}

\newtheorem{introconj}[introthm]{Conjecture}

\makeatletter
\let\save@mathaccent\mathaccent
\newcommand*\if@single[3]{%
  \setbox0\hbox{${\mathaccent"0362{#1}}^H$}%
  \setbox2\hbox{${\mathaccent"0362{\kern0pt#1}}^H$}%
  \ifdim\ht0=\ht2 #3\else #2\fi
  }
\newcommand*\rel@kern[1]{\kern#1\dimexpr\macc@kerna}
\newcommand*\widebar[1]{\@ifnextchar^{{\wide@bar{#1}{0}}}{\wide@bar{#1}{1}}}
\newcommand*\wide@bar[2]{\if@single{#1}{\wide@bar@{#1}{#2}{1}}{\wide@bar@{#1}{#2}{2}}}
\newcommand*\wide@bar@[3]{%
  \begingroup
  \def\mathaccent##1##2{%
    \let\mathaccent\save@mathaccent
    \if#32 \let\macc@nucleus\first@char \fi
    \setbox\z@\hbox{$\macc@style{\macc@nucleus}_{}$}%
    \setbox\tw@\hbox{$\macc@style{\macc@nucleus}{}_{}$}%
    \dimen@\wd\tw@
    \advance\dimen@-\wd\z@
    \divide\dimen@ 3
    \@tempdima\wd\tw@
    \advance\@tempdima-\scriptspace
    \divide\@tempdima 10
    \advance\dimen@-\@tempdima
    \ifdim\dimen@>\z@ \dimen@0pt\fi
    \rel@kern{0.6}\kern-\dimen@
    \if#31
      \overline{\rel@kern{-0.6}\kern\dimen@\macc@nucleus\rel@kern{0.4}\kern\dimen@}%
      \advance\dimen@0.4\dimexpr\macc@kerna
      \let\final@kern#2%
      \ifdim\dimen@<\z@ \let\final@kern1\fi
      \if\final@kern1 \kern-\dimen@\fi
    \else
      \overline{\rel@kern{-0.6}\kern\dimen@#1}%
    \fi
  }%
  \macc@depth\@ne
  \let\math@bgroup\@empty \let\math@egroup\macc@set@skewchar
  \mathsurround\z@ \frozen@everymath{\mathgroup\macc@group\relax}%
  \macc@set@skewchar\relax
  \let\mathaccentV\macc@nested@a
  \if#31
    \macc@nested@a\relax111{#1}%
  \else
    \def\gobble@till@marker##1\endmarker{}%
    \futurelet\first@char\gobble@till@marker#1\endmarker
    \ifcat\noexpand\first@char A\else
      \def\first@char{}%
    \fi
    \macc@nested@a\relax111{\first@char}%
  \fi
  \endgroup
}
\makeatother

\def\H{\mathbb H}

\def\dist{\mathrm{dist}}
\def\diam{\mathrm{diam}(G)}
\def\evendist{\dist_{\mathrm{even}}}
\def\even{\mathrm{Even}}

\begin{document}
\begin{abstract}
We make partial progress towards a proof of Conjecture 189 of Written on the Wall II by showing that a connected graph $G$ satisfying $\max\{\mathrm{dist_{even}}(v):v\in V(G)\}\le d_2+1$, where $\mathrm{dist_{even}}(v)$ is the number of vertices at an even distance from $v$ and $d_2$ is the second smallest degree of $G$, is traceable whenever at least one of four conditions holds. These conditions involve the vertex-connectivity, order, and diameter of $G$.
\end{abstract}
\maketitle
\section{Introduction and Preliminaries}

Throughout this paper, the term \emph{graph} refers to a nontrivial simple undirected graph with finite vertex set, and $G$ will be used to denote a graph $(V(G),E(G))$. In cases where the graph under consideration is clear, we write the vertex and edge sets as $V$ and $E$ respectively. We write $xy$ for an edge joining vertices $x$ and $y$.

Write $n\coloneqq |V|$ and let $d_1\le \dots\le d_n$ be the degrees of the vertices of $G$. 

Denote the degree and the neighbourhood of $v\in V$ by $\deg_G(v)$ and $N_G(v)$ respectively, the distance between $x,y\in V$ by $\dist_G(x,y)$, and the set of vertices at a distance $k$ from $v$ by $N_k(v)$. In cases where the graph under consideration is clear, we omit it from the notation.

Denote the diameter of $G$ by $\diam$. In places where the notation $d_1$ may lead to ambiguity, we denote the minimum degree of $G$ by $\delta(G)$. The vertex-connectivity of $G$ (see \cite[Section 1.4]{Diestel}) is denoted by $\kappa(G)$.

A graph is said to be \emph{traceable} if it has a Hamiltonian path, and it is said to be \emph{Hamiltonian} if it has a Hamiltonian cycle.

For $v\in V$, define $\even_G(v)\coloneqq \{w\in V:\dist(v,w)\text{ is even}\}=\bigcup\limits_{k\ge 0} N_{2k}(v)$. Note that $v\in \even(v)$, since $\dist(v,v)=0$. Write $\evendist^G(v)\coloneqq |\even_G(v)|$. We omit $G$ when the graph under consideration is clear from context.

For $U\subseteq V$, denote by $G[U]$ the induced subgraph of $G$ with vertex set $U$.

We use notation pertaining to the addition or deletion of edges or vertices from \cite[Section 1.1]{Diestel}. For notation pertaining to paths, see \cite[Section 1.3]{Diestel}.

The following is Conjecture 189 from \emph{Written on the Wall II} (\cite{WOWII}), and in this paper we make partial progress towards its resolution.

\begin{introconj}\label{the conjecture}
Let $G=(V,E)$ be a connected graph such that 
\[\max\{\evendist(v):v\in V\}\le d_2+1.\] Then $G$ is traceable.
\end{introconj}

The theorem below summarises the results obtained in this paper.

\begin{introthm}\label{thm:main}
Let $G=(V,E)$ be a connected graph such that 
\[\max\{\evendist(v):v\in V\}\le d_2+1.\] Then $G$ is traceable if any of the following conditions holds:
\begin{enumerate}
\item $\diam\le 2$;
\item $|V(G)|\le2d_2+2$;
\item $|V(G)|=2d_2+3$ and $G$ is 2-connected; or
\item $G$ is not 2-connected.
\end{enumerate}
\end{introthm}

We work under the following hypothesis almost throughout the paper, and hence we label it for the sake of brevity:
\begin{equation}\tag{$\H$}\label{the hypothesis}
\text{ $G$ is a connected graph and }\max\{\evendist(v):v\in V\}\le d_2+1.
\end{equation}

The rest of the paper is organised as follows. In \S~\ref{sec 2}, we prove the first part of Theorem \ref{thm:main}, namely the diameter condition, and list some useful lemmas that will feature throughout the paper. In \S~\ref{sec 3}, we prove some ``closure lemmas'' that take inspiration from the methods presented in \cite{BondyChvatal} and are used throughout the last section. In \S~\ref{sec 4}, we apply the aforementioned closure lemmas to prove the second and third parts of Theorem \ref{thm:main}. In \S~\ref{sec 5}, we use the block graph of $G$ to conclude that a counterexample to Conjecture \ref{the conjecture} must be 2-connected, thus proving the fourth and last part of \ref{thm:main}.

\subsection*{Acknowledgements} The authors thank Chaitanya Mahawar for valuable discussions.

\section{Basic results}\label{sec 2}

\begin{rem}
Since any connected $G$ with $\le 3$ vertices is traceable, we may assume that $|V(G)|\ge 4$.
\end{rem}

The following is a direct consequence of \cite[Theorem 2]{Ore1960}.

\begin{lem}\label{lem:OreHam}
Let $G$ be a graph with $n\ge 3$. Then $G$ is traceable if every pair of distinct non-adjacent vertices $x,y\in V$ satisfies $\deg(x)+\deg(y)\ge n-1$.
\end{lem}

\begin{thm}\label{thm:diamle 2}
If $G$ satisfies \eqref{the hypothesis} and $\diam\le 2$ then $G$ is traceable.
\end{thm}
\begin{proof}
Let $M\coloneqq \max\{\evendist(v):v\in V\}$.

Since $\diam\le 2$, we have $V=\{v\}\sqcup N(v)\sqcup N_2(v)$ and $\even(v)=\{v\}\sqcup N_2(v)$ for any $v\in V$.

This gives $\evendist(v)=n-\deg(v)$.

We therefore have $M=n-d_1$.

Since $G$ satisfies \eqref{the hypothesis}, we have $M\le 1+d_2$.

This gives us $n-d_1\le 1+d_2$, or equivalently, $d_1+d_2\ge n-1$.

Let $x,y\in V$ be distinct. Then $\deg(x)+\deg(y)\ge d_1+d_2\ge n-1$. In particular, this holds for all pairs of distinct non-adjacent vertices $x,y$.

Hence by Lemma \ref{lem:OreHam}, we have that $G$ is traceable.

This proves (1) from Theorem \ref{thm:main}.
\end{proof}

\begin{lem}[Chv\'{a}tal]\cite[Exercise 4.2.4]{BondyMurty1976}\label{lem:Chvatal}
A graph $G$ is traceable if for each $i<\frac{n+1}{2}$, we have $d_i\ge i$ or $d_{n-i+1}\ge n-i$.
\end{lem}

\begin{lem}\label{lem:nle 2d+1}
If $G$ is connected and $n\le 2d_2+1$ then $G$ is traceable.
\end{lem}
\begin{proof}
Since $G$ is connected, we have $d_1\ge 1$.

For $2\le i<\frac{n+1}{2}$, we have $i\le \frac{n-1}{2}\le d_2\le d_i$, and hence $d_i\ge i$.

We therefore have $d_i\ge i$ for each $i<\frac{n+1}{2}$.

By Lemma \ref{lem:Chvatal}, $G$ is traceable.
\end{proof}

\begin{lem}\label{lem:G+z Hamiltonian}
A graph $G$ is traceable if and only if $G+z$ is Hamiltonian, where $z$ is such that $vz\in E(G+z)$ for each $v\in V(G)$.
\end{lem}
\begin{proof}
If $P$ is a Hamiltonian path in $G$, then joining its two endpoints to $z$ gives a Hamiltonian cycle in $G+z$. Conversely, deleting $z$ from a Hamiltonian cycle in $G+z$ gives a Hamiltonian path in $G$.
\end{proof}

\section{Closure lemmas}\label{sec 3}

\begin{lem}\label{lem:ordinary-path-closure}
Let $G$ have vertices $x,y$ such that $xy\notin E(G)$ and $\deg_G(x)+\deg_G(y)\ge n-1$. Then $G$ is traceable if and only if $G+xy$ is traceable.
\end{lem}
\begin{proof}
If $G$ is traceable then $G+xy$ is clearly traceable.

We now prove the converse. Add a vertex $z$ to $G$ such that $vz\in E(G+z)$ for every $v\in V(G)$.

Observe that $\deg_{G+z}(x)+\deg_{G+z}(y)\ge n-1+2=n+1=|V(G+z)|$.

Since $G+xy$ is traceable, we have by Lemma \ref{lem:G+z Hamiltonian} that $G+z+xy$ is Hamiltonian. By \cite[(2.1)]{BondyChvatal}, we therefore have that $G+z$ is Hamiltonian. Again using Lemma \ref{lem:G+z Hamiltonian}, we obtain that $G$ is traceable.
\end{proof}

\begin{lem}\label{lem:distance-3-closure}
Let $G$ be connected. Let $x,y\in V(G)$ be such that $xy\notin E(G)$, $\deg_G(x)+\deg_G(y)\ge n-2$, and $\dist_G(x,y)\ge 3$. Then $G$ is traceable if and only if $G+xy$ is traceable.
\end{lem}
\begin{proof}
If $G$ is traceable then $G+xy$ is clearly traceable.

We now prove the converse. Add a vertex $z$ to $G$ such that $vz\in E(G+z)$ for every $v\in V(G)$. Write $K\coloneqq G+z$ and $M\coloneqq n+1=|V(K)|$.

Suppose that $K+xy$ is Hamiltonian. Consider a Hamiltonian cycle $C$ in $K+xy$. If $C$ does not contain $xy$, then $K$ is Hamiltonian and hence $G$ is traceable.

If $C$ contains $xy$, then deleting $xy$ from $C$ gives a Hamiltonian path $P\coloneqq p_1\dots p_M$ with $p_1=x$ and $p_M=y$.

For $1\le i\le M-1$, define
\[
A\coloneqq \{i:xp_{i+1}\in E(K)\} \text{ and } B\coloneqq \{i:yp_i\in E(K)\}.
\]
Clearly, $|A|=|N_K(x)|=\deg_K(x)$ and $|B|=|N_K(y)|=\deg_K(y)$.

If $i\in A\cap B$, then
\[
p_1\dots p_i p_Mp_{M-1}\dots p_{i+1}p_1
\]
is a Hamiltonian cycle in $K$, implying that $G$ is traceable.

Now suppose that $A\cap B=\emptyset$.

Then 
\[|A\cup B|=|A|+|B|=\deg_K(x)+\deg_K(y)=\deg_G(x)+\deg_G(y)+2\ge n=M-1\]
and hence $A\cup B=\{1,\dots,M-1\}$.

Observe that $1\in A$, $M-1\in B$, and whenever $i\in A$ and $i+1\in B$, we have $p_{i+1}\in N_K(x)\cap N_K(y)$. Since $\dist_G(x,y)\ge 3$, we have that $N_G(x)\cap N_G(y)=\emptyset$, and hence that $N_K(x)\cap N_K(y)=\{z\}$.

There is therefore a unique $i$ such that $i\in A$ and $i+1\in B$. Let $j$ be the least element in $B$. Since $1\in A$, we have $1\ne j$ and hence $j-1\in A$. Let $j'$ be the greatest element in $A$. Since $M-1\in B$, we have $M-1\ne j'$ and hence $j'+1\in B$.  If $j'\ne j-1$, then $j'$ and $j-1$ are two distinct values of $i$ for which $i\in A$ and $i+1\in B$, which is a contradiction. Hence $j'=j-1$. Using the fact that $A$ and $B$ partition $\{1,\dots,M-1\}$, we obtain that $A=\{1,\dots,j-1\}$ and $B=\{j,\dots,M-1\}$. We also obtain that $p_j=z$.

Write $X\coloneqq \{p_1,\dots, p_{j-1}\}$ and $Y\coloneqq \{p_{j+1},\dots,p_M\}$. Note that $X$ and $Y$ partition $V(G)$.

Consider the induced subgraph $G[X]$ of $G$. Since $\{p_2,\dots,p_{j-1}\}\subset V(G)$, $G[X]$ contains the path $p_2\dots p_{j-1}$.

We have $\{1,\dots,j-2\}\subset \{1,\dots,j-1\}=A=\{i:xp_{i+1}\in E(K)\}$, and hence $xp_i\in E(K)$ for $2\le i\le j-1$. Since $z\notin \{p_2,\dots,p_{j-1}\}$, we have $xp_i\in E(G)$ for $2\le i\le j-1$, which gives $xp_i\in E(G[X])$ for $2\le i\le j-1$. Hence given any vertex $p_k$ with $1\le k\le j-1$, we have a Hamiltonian path
\[
p_kp_{k-1}\dots p_2xp_{k+1}p_{k+2}\dots p_{j-1}
\]
in $G[X]$ with $p_k$ as an endpoint. 

Now consider $G[Y]$. Since $\{p_{j+1},\dots,p_{M-1}\}\subset V(G)$, $G[Y]$ contains the path $p_{j+1}\dots p_{M-1}$.

We have $\{j+1,\dots,M-1\}\subset \{j,\dots,M-1\}=B=\{i:yp_i\in E(K)\}$, and hence $yp_i\in E(K)$ for $j+1\le i\le M-1$. Since $z\notin \{p_{j+1},\dots,p_{M-1}\}$, we have $yp_i\in E(G)$ for $j+1\le i\le M-1$, which gives $yp_i\in E(G[Y])$ for $j+1\le i\le M-1$. Hence given any vertex $p_k$ with $j+1\le k\le M-1$, we have a Hamiltonian path
\[
p_k p_{k+1}\dots p_{M-1}y p_{k-1}\dots p_{j+1}
\]
in $G[Y]$ with $p_k$ as an endpoint.

Since $G$ is connected, there exists $ab\in E(G)$ such that $a\in X$ and $b\in Y$. Choose Hamiltonian paths $P_1$ $P_2$ in $G[X]$ and $G[Y]$ respectively with $a$ an endpoint of $P_1$ and $b$ an endpoint of $P_2$. Concatenating $P_1$ and $P_2$ gives a Hamiltonian path in $G$, thus proving that $G$ is traceable.
\end{proof}

\begin{lem}\label{lem:two-connected lemmas}
Let $G$ be a 2-connected graph with $|V(G)|=n$ and $x,y\in V(G)$ such that $xy\notin E(G)$. Then:
\begin{enumerate}
\item If $N_G(x)\cap N_G(y)=\emptyset$ and $\deg_G(x)+\deg_G(y)\ge n-3$, then $G$ is traceable if and only if $G+xy$ is traceable.
\item If $|N_G(x)\cap N_G(y)|\le 1$ and $\deg_G(x)+\deg_G(y)\ge n-2$, then $G$ is traceable if and only if $G+xy$ is traceable.
\end{enumerate}
\end{lem}
\begin{proof}
If $G$ is traceable then $G+xy$ is traceable. Hence in both cases it suffices to prove the converse.

We now suppose that $G+xy$ is traceable. Add a vertex $z$ to $G$ such that $vz\in E(G+z)$ for every $v\in V(G)$. Write $K\coloneqq G+z$ and $M\coloneqq n+1=|V(K)|$. Since $G+xy$ is traceable, Lemma \ref{lem:G+z Hamiltonian} gives us that $K+xy$ is Hamiltonian. Consider a Hamiltonian cycle $C$ in $K+xy$. If $C$ does not contain $xy$, then $K=G+z$ is Hamiltonian, and hence $G$ is traceable by Lemma \ref{lem:G+z Hamiltonian}. Now suppose $C$ contains $xy$. Deleting $xy$ from $C$ gives a Hamiltonian path $P\coloneqq p_1\dots p_M$ with $p_1=x$ and $p_M=y$.

For $1\le i\le M-1$, define $A\coloneqq \{i:xp_{i+1}\in E(K)\}$ and $B\coloneqq \{i:yp_i\in E(K)\}$. Clearly, $|A|=|N_K(x)|=\deg_K(x)$ and $|B|=|N_K(y)|=\deg_K(y)$.

If $i\in A\cap B$, then
\[
p_1\dots p_i p_Mp_{M-1}\dots p_{i+1}p_1
\]
is a Hamiltonian cycle in $K$, implying that $G$ is traceable.

Now suppose that $A\cap B=\emptyset$. Then $|A\cup B|=|A|+|B|=\deg_K(x)+\deg_K(y)=\deg_G(x)+\deg_G(y)+2$.

We will say that $1\le i\le M-1$ is \emph{labelled} $A$ or $B$ if it belongs to $A$ or $B$ respectively. If $i\notin A\cup B$ then $i$ is said to have no label.

In case (1), we have $|A\cup B|\ge \deg_G(x)+\deg_G(y)+2\ge n-3+2=n-1$, and hence there is at most one unlabelled position.

In case (2), we have $|A\cup B|\ge \deg_G(x)+\deg_G(y)+2\ge n-2+2=n$, and hence all positions are labelled.

In both cases, 1 is labelled $A$ and $M-1$ is labelled $B$. 

From the definitions of $A$ and $B$, it is apparent that whenever $i\in A$ and $i+1\in B$, we have $p_{i+1}\in N_K(x)\cap N_K(y)$. This means that there is exactly one such $i$ in case (1) (with $p_{i+1}=z$) and at most one other such $i$ in case (2). Let $A^*$ and $B^*$ denote a (nonempty) run of vertices from $A$ and $B$ respectively. Let $\square$ denote the unlabelled position, wherever applicable. Since $p_{i+1}=z$ is adjacent to both $x$ and $y$, we have that $i$ is labelled $A$ and $i+1$ is labelled $B$. 

The possible run patterns for $P$, up to reversing the path and interchanging $A$ and $B$, are:
\begin{table}[h]
\centering
\begin{tabular}{c|l}
\textbf{Case} & \textbf{Run pattern} \\ \hline
(1) &
$A^*B^*$,\ $A^*B^*\square B^*$,\ $A^*B^*A^*\square B^*$ \\[2mm]
(2) &
$A^*B^*$,\ $A^*B^*A^*B^*$
\end{tabular}
\caption{List of possible run patterns.}
\label{table:run patterns}
\end{table}

We now present a general observation that will allow us to construct a Hamiltonian path in each case. Suppose a graph contains a path $Q\coloneqq w_0w_1\dots w_mh$ such that $h$ is not adjacent to at most one of $w_0,w_1,\dots,w_m$. We claim that there exists a Hamiltonian path of $G[V(Q)]$ ending at any given vertex, except when $hw_0\notin E(Q)$, in which case we can guarantee the existence of a Hamiltonian path of $G[V(Q)]$ ending at any given vertex except $w_1$. In all cases, $Q$ itself is a Hamiltonian path of $G[V(Q)]$ ending at $h$. Suppose that $h$ is adjacent to all other vertices in $Q$. For $0\le r<m$,
\[
w_rw_{r-1}\dots w_0 hw_{r+1}\dots w_m
\]
is a Hamiltonian path of $G[V(Q)]$ ending at $w_r$ and $w_m$. Now suppose that $h$ is adjacent to all vertices in $Q$ except $w_j$ for some given $0\le j\le m$. First consider the case where $j\ge 1$. In this case, for $r\in\{0,1,\dots,m-1\}\setminus\{j-1\}$,
\[
w_rw_{r-1}\dots w_0 hw_{r+1}\dots w_m
\]
is a Hamiltonian path of $G[V(Q)]$ ending at $w_r$ and $w_m$. When $j=1$, the path
\[
w_1w_2\dots w_mhw_0
\]
has $w_{j-1}=w_0$ as an endpoint, and when $j\ge 2$, the path
\[
w_0w_1\dots w_{j-2}hw_mw_{m-1}\dots w_{j-1}
\]
has $w_{j-1}$ as an endpoint. If $j=0$ and $r\ge 2$, the path
\[
w_0w_1\dots w_{r-1}hw_mw_{m-1}\dots w_r
\] has $w_r$ as an endpoint, and $Q$ has $w_0$ as an endpoint.

We now consider the run patterns listed in Table \ref{table:run patterns}:

\begin{itemize}
\item[(a)] $A^*B^*$ (cases (1) and (2)):
Write $P=xp_1'\dots p_r'zq_1\dots q_s y$. Define $L\coloneqq \{x,p_1',\dots,p_r'\}$ and $R\coloneqq \{q_1,\dots,q_y,s\}$. Then $x$ is adjacent to every other vertex in $L$ and $y$ is adjacent to every other vertex in $R$. Hence both $G[L]$ and $G[R]$ have Hamiltonian paths ending at any given vertex. Since $G$ is connected, there exists $lr\in E(G)$ such that $l\in L$ and $r\in R$. Concatenate Hamiltonian paths of $G[L]$ and $G[R]$ ending at $l$ and $r$ respectively along $lr$ to obtain a Hamiltonian path in $G$.

\item[(b)] $A^*B^*\square B^*$ (case (1)):
Write $P=xp_1'\dots p_{a-1}'zu_1\dots u_bw_0\dots w_{c-1}y$ ($a,b,c\ge 1$), and suppose that $y$ is not adjacent to $u_b$ (to account for the unlabelled position). Define $L\coloneqq \{x,p_2,\dots,p_a\}$ and $R\coloneqq V(G)\setminus L$. Then $x$ is adjacent to every other vertex in $L$, implying that $G[L]$ has a Hamiltonian path ending at any given vertex. We have that $y$ is adjacent to $u_1,\dots u_{b-1},w_0,\dots, w_{c-1}$. If $b\ge 2$, then as seen above, $G[R]$ has a Hamiltonian path ending at any given vertex. If $b=1$, then $G[R]$ has a Hamiltonian path ending at any given vertex except possibly at $w_0$. If there exists $lr\in E(G)$ with $l\in L$ and $r\in R\setminus \{w_0\}$, then we concatenate Hamiltonian paths of $G[L]$ and $G[R]$ ending at $l$ and $r$ respectively along $lr$ to obtain a Hamiltonian path in $G$. Now suppose that every edge between a vertex in $L$ and a vertex in $R$ is incident with $w_0$. Then $G-w_0$ is not connected, since it contains no path from any vertex of $L$ to any vertex of $R\setminus\{w_0\}$. This contradicts the fact that $G$ is 2-connected. A similar argument works for the pattern $A^*\square A^* B^*$, since it is obtained by reversing $P$ and interchanging $A$ and $B$ in the pattern $A^*B^*\square B^*$.

\item[(c)] $A^*B^*A^*B^*$ (case (2)) and $A^*B^*A^*\square B^*$ (case (1)):
Write 
\[P=xp_1'\dots p_{a-1}'zu_1\dots u_{b-1}q_0\dots q_c w_1\dots w_my.\] 
For the $A^*B^*A^*B^*$ pattern, which falls under case (2), we may assume that the first $A-B$ transition occurs at $z$, by reversing $P$ and interchanging $A$ and $B$ if necessary. In case (2), the second $A-B$ transition occurs at $q_c$, and $yq_c\in E(G)$. In case (1), $yq_c\notin E(G)$, accounting for the unlabelled position. In both cases, all edges of $P$, except those incident with $z$, are edges of $G$, and we have $xp_i',xq_i,yu_i,yw_i\in E(G)$ for the respective index ranges.

Suppose $b\ge 2$. If $c=1$, then
\[
p_1'\dots p_{a-1}'xq_1q_0u_{b-1}u_{b-2}\dots u_1yw_mw_{m-1}\dots w_1
\] is a Hamiltonian path in $G$.

If $c\ge 2$, then
\[
p_1'\dots p_{a-1}'xq_1q_0u_{b-1}u_{b-2}\dots u_1yw_mw_{m-1}\dots w_1q_cq_{c-1}\dots q_2
\] is a Hamiltonian path in $G$. If $m=0$ in the above path, then $q_c$ occurs directly after $y$.

Now suppose $b=1$. In this case, the vertices $u_1,\dots,u_{b-1}$ do not occur, and hence
\[
P=xp_1'\dots p_{a-1}'zq_0\dots q_c w_1\dots w_my.
\]
Then $q_0$ has neighbours $z$ and $q_1$ in $P$. Since $z\notin V(G)$, the only neighbour of $q_0$ in $P$ among the vertices of $G$ is $q_1$. If $xq_0\in E(G)$ or $yq_0\in E(G)$ then we would obtain a doubly labelled position, which is a contradiction. Since $G$ is 2-connected, we have $\deg_G(q_0)\ge 2$ and hence that $q_0$ has a neighbour $v\ne q_1$ in $G$. The only possibilities for this neighbour are $p_i'$ for some $1\le i\le a-1$, $q_i$ for some $2\le i\le c$, or $w_i$ for some $1\le i\le m$.

If $v=p_i'$ for some $1\le i\le a-1$, then
\[
p_1'\dots p_{i-1}'xp_{a-1}p_{a-2}\dots p_iq_0q_1\dots q_cw_1\dots w_my
\] is a Hamiltonian path in $G$.

If $v=q_i$ for some $2\le i\le c$, then
\[
p_1'\dots p_{a-1}'xq_{i-1}q_{i-2}\dots q_0q_iq_{i+1}\dots q_cw_1\dots w_my
\] is a Hamiltonian path in $G$.

If $v=w_1$, then
\[
p_1'\dots p_{a-1}'xq_cq_{c-1}\dots q_1q_0w_1\dots w_my
\] is a Hamiltonian path in $G$.

Now suppose $v=w_i$ for some $2\le i\le m$. If $c=1$, then
\[
p_1'\dots p_{a-1}'xq_1q_0w_iw_{i+1}\dots w_myw_{i-1}w_{i-2}\dots w_1
\] is a Hamiltonian path in $G$.

If $c\ge 2$, then
\[
p_1'\dots p_{a-1}'xq_1q_0w_iw_{i+1}\dots w_myw_{i-1}w_{i-2}\dots w_1q_cq_{c-1}\dots q_2
\] is a Hamiltonian path in $G$.

A similar argument works for the pattern $A^*\square B^* A^*B^*$, since it is obtained by reversing $P$ and interchanging $A$ and $B$ in the pattern $A^*B^*A^*\square B^*$.
\end{itemize}
\end{proof}

\section{Applying the closure lemmas}\label{sec 4}

\begin{thm}\label{thm:2d+2}
If $G$ satisfies \eqref{the hypothesis} and $n= 2d_2+2$ then $G$ is traceable.
\end{thm}
\begin{proof}
Start with $G$ and add an edge between non-adjacent vertices $u,v\in V(G)$ if they satisfy either of the following conditions (with respect to the current graph at that stage):
\begin{enumerate}
\item $\deg(u)+\deg(v)\geq n-1$ 
\item  $\deg(u)+\deg(v)\geq n-2$ and $\dist(u,v) \geq 3$
\end{enumerate}

If there exists $z\in V(G)$ with $\deg_G(z)<d_2$, we impose the additional restriction that no added edge is incident with $z$. By lemmas \ref{lem:ordinary-path-closure}  and \ref{lem:distance-3-closure}, every edge addition preserves traceability in both directions. Let $H$ be the terminal graph obtained after carrying out the procedure described above, where terminality refers to the fact that there exists no pair $u,v$ of non-adjacent vertices in $H$ satisfying either of the given conditions.

Suppose that $H$ is not traceable and let $a_1\le a_2\le \dots \le a_n$ be its degree sequence. We have
\[
\frac{n+1}{2}=\frac{2d_2+3}{2}=d_2+\frac{3}{2},
\]
and hence that $1\le i<\frac{n+1}{2}$ is equivalent to $1\le i\le d_2+1$. We have $a_1\ge 1$ since $H$ is connected, and $a_i\ge d_i\ge d_2\ge i$ for $2\le i\le d_2$. Since $H$ is not traceable, by Lemma \ref{lem:Chvatal}, we have $a_{d_2+2}\le d_2$, implying that at least $d_2+2$ vertices of $H$ have degree $\le d_2$.

Suppose that $G$ has no vertex of degree less than $d_2$. Let
\[
S\coloneqq \{v\in V(H):\deg_H(v)\le d_2\} \text{ and } R\coloneqq V(H)\setminus S.
\]
Since $G$ has no vertex of degree less than $d_2$, neither does $H$. Hence
\[
S=\{v\in V(H):\deg_H(v)=d_2\}.
\]
If $u\in S$ and $v\in R$ are non-adjacent, then $\deg_H(u)+\deg_H(v) \geq d_2+(d_2+1)=n-1$, which contradicts the assumption that $H$ is terminal. Thus every vertex of $S$ is adjacent to every vertex of $R$ in $H$. Let $r\coloneqq |R|$. Then any $u\in S$ has $d_2-r$ neighbours in $S$ and hence 
\[
(|S|-1)-(d_2-r)=n-r-1-d_2+r=d_2+1
\]
non-neighbours in $S$ (other than itself). These $d_2+1$ non-neighbours each have degree $d_2$. Consider any such non-neighbour $v$. We have $\deg_H(u)+\deg_H(v)=2d_2=n-2$, and hence the terminality of $H$ gives $\dist_H(u,v)=2$. For any $s\in S$, we have $\deg_G(s)\le \deg_H(s)=d_2$. Since $G$ has no vertices of degree less than $d_2$, we must have $\deg_G(s)=d_2$ and hence that no added edges are incident with any vertex of $S$. We have $\dist_G(u,v)\ge \dist_H(u,v)=2$. Consider a path $uwv$ in $H$. Since no added edges are incident with $u$ or $v$, we must have $uw,wv\in E(G)$ and hence $\dist_G(u,v)=2$. We therefore have $\evendist^G(u)\ge d_2+2$ ($u$ itself and $d_2+1$ other non-neighbours in $S$), contradicting \eqref{the hypothesis}.

Now suppose that $G$ has a vertex $z$ of degree less than $d_2$. By construction, no added edges are incident with $z$ and hence $\deg_H(z)=\deg_G(z)$. Let
\[
S\coloneqq \{v\in V(H):\deg_H(v)\le d_2\}\setminus \{z\} \text{ and } R\coloneqq \{v\in V(H):\deg_H(v)>d_2\}.
\]
Every $u\in V(G)\setminus \{z\}$ satisfies $\deg_H(u)\ge \deg_G(u)\ge 2$, and hence
\[
S=\{v\in V(H):\deg_H(v)= d_2\}.
\]
If $u\in S$ and $v\in R$ are non-adjacent, then $\deg_H(u)+\deg_H(v) \geq d_2+(d_2+1)=n-1$, which contradicts the assumption that $H$ is terminal. Thus every vertex of $S$ is adjacent to every vertex of $R$ in $H$. Let $r\coloneqq |R|$. Consider $u\in S$ such that $uz\in E(H)$. Then $u$ has $d_2-r-1$ neighbours in $S$ and hence
\[
(|S|-1)-(d_2-r-1)=(n-r-1-1)-(d_2-r-1)=n-r-2-d_2+r+1=n-d_2-1=d_2+1
\]
non-neighbours in $S$ (other than itself). Consider any such non-neighbour $v$. We have $\deg_H(u)+\deg_H(v)=2d_2=n-2$, and hence the terminality of $H$ gives $\dist_H(u,v)=2$. For any $s\in S$, we have $\deg_G(s)\le \deg_H(s)=d_2$. Since $z\notin S$, we have $\deg_G(s)\ge d_2$, implying that $\deg_G(s)= d_2$ and hence that no added edges are incident with any vertex of $S$. We have $\dist_G(u,v)\ge \dist_H(u,v)=2$. Consider a path $uwv$ in $H$. Since no added edges are incident with $u$ or $v$, we must have $uw,wv\in E(G)$ and hence $\dist_G(u,v)=2$. We therefore have $\evendist^G(u)\ge d_2+2$ ($u$ itself and $d_2+1$ other non-neighbours in $S$), contradicting \eqref{the hypothesis}. This shows that no vertex in $S$ is adjacent to $z$ in $H$. Every $u\in S$ therefore has $d_2-r$ neighbours in $S$ and hence $d_2$ non-neighbours in $S$ (other than itself). As seen above, these $d_2$ non-neighbours in $S$ are all at distance two from $u$ in $G$. Since $G$ is connected, $z$ has a neighbour $w\in R$ (in $G$). Since $w\in R$, we have $uw\in E(H)$. Since $u\in S$, no added edge is incident with it and hence $uw\in E(G)$. Thus $\dist_G(u,z)=2$ via the path $uwz$ in $G$. We therefore have $\evendist^G(u)\ge d_2+2$, contradicting \eqref{the hypothesis}.

Hence the assumption that $H$ is not traceable leads to a contradiction, showing that $H$ is traceable. By lemmas \ref{lem:ordinary-path-closure} and \ref{lem:distance-3-closure}, we have that $G$ is traceable.

This proves (2) from Theorem \ref{thm:main}.
\end{proof}

\begin{thm}\label{thm:2d+3}
If a 2-connected graph $G$ satisfies \eqref{the hypothesis} and $n\le 2d_2+3$ then $G$ is traceable.
\end{thm}
\begin{proof}
By Lemma \ref{lem:nle 2d+1} and Theorem \ref{thm:2d+2}, it suffices to check the case where $n=2d_2+3$. Consider $v\in V(G)$ with $\deg_G(v)=d_2$. By \eqref{the hypothesis}, we have that $n-\evendist^G(v)\ge n-d_2-1=d_2+2$. Since $\deg_G(v)=d_2$, we obtain that $v$ has at least two non-neighbours at odd distance $\ge 3$ in $G$. At most one such non-neighbour has degree less than $d_2$, and hence there exists $w\in V(G)$ such that $\dist_G(v,w) \ge 3$ and $\deg_G(w) \ge d_2$.

We add edges while preserving traceability to incrementally form a graph $H$. Let $H$ initially be $G$. Whenever $v\in V(G)$ satisfies $\deg_G(v)=\deg_H(v)=d_2$, choose $w$ with $\dist_G(v,w) \ge 3$ and $\deg_G(w) \ge d_2$. Then $N_G(v)\cap N_G(w)=\emptyset$. Since $\deg_G(v)=\deg_H(v)$, we have that no added edge is incident with $v$. Every common neighbour of $v$ and $w$ in $H$ must therefore have arisen from an added edge incident with $w$.

There are three cases:
\begin{enumerate}
\item If $\deg_H(w)=d_2$, then since $\deg_G(w)\ge d_2$, we have $\deg_G(w)=d_2$. Hence no added edges are incident with $w$, implying that $N_H(v)\cap N_H(w)=\emptyset$. We also have $\deg_H(v)+\deg_H(w) = 2d_2=n-3$, and hence by case (1) of Lemma \ref{lem:two-connected lemmas}, $H+vw$ is traceable if and only if $H$ is traceable.
\item If $\deg_H(w)=d_2+1$, then since $\deg_G(w)\ge d_2$, we have that at most one added edge is incident with $w$, thus showing that $|N_H(v)\cap N_H(w)|\le 1$. We also have $\deg_H(v)+\deg_H(w)=2d_2+1=n-2$, and hence by case (2) of Lemma \ref{lem:two-connected lemmas}, $H+vw$ is traceable if and only if $H$ is traceable.
\item If $\deg_H(w)\geq d_2+2$ then $\deg_H(v)+\deg_H(w) \ge 2d_2+2=n-1$, and hence by Lemma \ref{lem:ordinary-path-closure}, $H+vw$ is traceable if and only if $H$ is traceable.
\end{enumerate}
Thus $vw$ can always be added to $H$ while preserving traceability, increasing the degree of $v$ to $d_2+1$. Repeating the operation eventually produces a graph $H$ in which every $v\in V(G)$ with $\deg_G(v)=d_2$ has $\deg_H(v)\ge d_2+1$. Hence at most one vertex of $H$ has degree less than $d_2+1$.
Let $a_1\le a_2 \le \dots \le a_n$ be the degree sequence of $H$. Then $a_1 \geq 1$ and $a_i\ge d_2+1\ge i$ for $2\le i\le d_2+1$. Hence we have $a_i\ge i$ for all $1\le i<d_2+2=\frac{n+1}{2}$. Applying Lemma \ref{lem:Chvatal} gives that $H$ is traceable. Since each edge addition preserved traceability, we have that $G$ is traceable.

This proves (3) from Theorem \ref{thm:main}.
\end{proof}

\section{The block graph}\label{sec 5}

\begin{defn}\cite[Sections 1.3, 1.4]{Diestel}
For $A,B\subseteq V$, a path $P=x_0\dots x_k$ is said to be an \emph{$A-B$ path} if $V(P)\cap A=\{x_0\}$ and $V(P)\cap B=\{x_k\}$. If $X\subseteq V\cup E$ and $A,B\subseteq V$ are such that every $A-B$ path in $G$ contains a vertex or an edge from $X$, then $X$ is said to \emph{separate} $A$ and $B$ in $G$. If $X$ separates the sets $\{a\},\{b\}\subseteq V$ and $a,b\notin X$, then $X$ is said to \emph{separate} the vertices $a$ and $b$ in $G$. If $X=\{v\}\subseteq V$ separates $A$ and $B$ in $G$, then the vertex $v$ is said to \emph{separate} $A$ and $B$ in $G$. A vertex that separates two vertices of the same component in $G$ is called a \emph{cutvertex}. If $X=\{e\}\subseteq E$ separates $A$ and $B$ in $G$, then the edge $E$ is said to \emph{separate} $A$ and $B$ in $G$. An edge that separates its endpoints is called a \emph{bridge}.
\end{defn}

\begin{lem}\label{lem:bridge-evendist}
If $G$ is connected and $uv$ is a bridge in $G$, then $\evendist(u)+\evendist(v)=n$.
\end{lem}
\begin{proof}
Since $uv$ is a bridge, deleting it yields a graph with two components, one containing $u$ and the other containing $v$. If $x$ is in the component containing $u$, then every $v-x$ path in $G$ uses the edge $uv$, and hence $\dist_G(v,x)=\dist_G(u,x)+1$. If $y$ is in the component containing $u$, then every $u-y$ path in $G$ uses the edge $uv$, and hence $\dist_G(u,y)=\dist_G(v,y)+1$. Hence every vertex in $G$ is at an even distance from exactly one of $u$ and $v$. 
\end{proof}

\begin{lem}\label{lem:nobridge}
If $G$ satisfies \eqref{the hypothesis} and is not traceable, then $G$ does not have a bridge.
\end{lem}
\begin{proof}
Suppose $G$ satisfies \eqref{the hypothesis} and has a bridge $uv$. Then Lemma \ref{lem:bridge-evendist} and \eqref{the hypothesis} give $n=\evendist(u)+\evendist(v)\le 2d_2+2$. From Theorem \ref{thm:2d+2}, we have that $G$ is traceable.
\end{proof}

\begin{defn}\cite[Section 3.1]{Diestel}
A \emph{block} in a graph $G$ is a maximal connected subgraph $B$ of $G$ without a cutvertex (note that $B$ does not have cutvertices \emph{of its own}, but it may contain a cutvertex \emph{of $G$}).
\end{defn}

\begin{rem}\cite[Section 3.1]{Diestel}
Every block in a graph $G$ is a maximal 2-connected subgraph of $G$, a bridge (with its ends), or an isolated vertex.
\end{rem}

We now define the \emph{block graph} of a graph $G$.

\begin{defn}\cite[Section 3.1]{Diestel}
Let $G$ be a graph, $A$ be the set of cutvertices of $G$, and $\mathcal{B}$ be the set of blocks of $G$. The graph on vertex set $A\cup \mathcal{B}$ with edge set $\{aB:B\in\mathcal{B}, a\in B\}$ is called the \emph{block graph} of $G$.
\end{defn}

\begin{rem}
If the block graph of $G$ is a tree, then we will refer to the blocks corresponding to leaves as \emph{end-blocks}.
\end{rem}

\begin{thm}\label{thm:block-cutvertex path}
If $G$ satisfies \eqref{the hypothesis} and is not traceable, then the block graph of $G$ is a path.
\end{thm}
\begin{proof}
From \cite[Lemma 3.1.4]{Diestel}, we know that the block graph of $G$ is a tree. Let $B_1,\dots,B_q$ be its end-blocks.

For $1\le i\le q$, let $c_i$ be the cutvertex in $B_i$, and write $X_i\coloneqq V(B_i)\setminus \{c_i\}$. By Lemma \ref{lem:nobridge}, $G$ does not have a bridge. We therefore have $|B_i|\ge 3$ and hence $|X_i|\ge 2$. Each $X_i$ therefore contains a vertex of degree at least $d_2$, since at most one vertex in $G$ has degree less than $d_2$.

Consider $x\in X_i$ such that $\deg_G(x)\ge d_2$. We have
\begin{equation}\label{eqn:neighbourhood-contained}
N_G(x)\subseteq (X_i \cup \{c_i\})\setminus \{x\} \text{ for $x\in X_i$, $1\le i\le q$}
\end{equation} and hence
\begin{equation}\label{eqn:block-size}
|X_i|\ge d_2 \text{ for } 1\le i\le q.
\end{equation}

Let
\begin{equation}\label{eqn:a-i in bgpath}
a_i\coloneqq |\{x\in X_i:\dist_G(x,c_i)\text{ is even}\}|.
\end{equation}

Consider $x_i \in X_i \cap N_G(c_i)$. For any $j\ne i$, every path from $x_i$ to any vertex in $X_j$ passes through $c_i$. This implies that $\dist_G(x_i,v)$ and $\dist_G(c_i,v)$ have opposite parity for every $v\in X_j$. This gives
\[
\left(\sum\limits_{j\ne i} |X_j|\right) +a_i+2 \le \evendist(c_i)+\evendist(x_i)\le 2d_2+2
\], from which we obtain
\begin{equation}\label{eqn:number-of-blocks-ineq}
\left(\sum\limits_{j\ne i} |X_j|\right) +a_i\le 2d_2.
\end{equation}

Inequalities \eqref{eqn:block-size} and \eqref{eqn:number-of-blocks-ineq} immediately yield $q\le 3$. 

We now consider the case $q=3$, which forces $|X_i|=d_2$ and $a_i=0$ for $1\le i\le 3$.

There is at most one vertex in $G$ with degree less than $d_2$. Suppose such a vertex exists, and call it $u$. For $1\le i\le 3$, every $x\in X_i$ except $u$ has degree at least $d_2$, by the definition of $d_2$. By \eqref{eqn:neighbourhood-contained}, we have
\[
\deg_G(x)=|N_G(x)|\le |(X_i \cup \{c_i\})\setminus \{x\}|=d_2,
\]
and hence $\deg_G(x)=d_2$ for all $x\in X_i\setminus \{u\}$. This implies that every $x\in X_i\setminus \{u\}$ is adjacent to every other vertex in $X_i \cup \{c_i\}$. Say that $u\in X_j$ for some $1\le j\le 3$. Since every $x\in X_j\setminus \{u\}$ is adjacent to every other vertex in $X_j \cup \{c_j\}$, we have that $u$ is adjacent to every other vertex in $X_j$, giving it a degree of at least $d_2-1$. If $u$ were also adjacent to $c_j$, then it would have a degree of $d_2$, which is a contradiction. If $u$ is not adjacent to $c_j$, then $\dist_G(u,c_j)=2$, since $uxc_j$ is a path for any $x\in X_j\setminus \{u\}$. By Equation \eqref{eqn:a-i in bgpath}, $\dist_G(u,c_j)=2$ implies that $a_j\ge 1$, which is a contradiction. Hence there is no vertex in $G$ with degree less than $d_2$. This implies that each $B_i$ is a clique $K_{d_2+1}$.

For fixed $i$, each $x\in X_j$ for $j\ne i$ satisfies $\dist_G(x,c_i)=\dist_G(c_j,c_i)+1$. Let $\{i,j,k\}=\{1,2,3\}$. Suppose $\dist_G(x,c_i)$ for $x\in X_j$ and $\dist_G(y,c_i)$ for $y\in X_k$ had the same parity. Then one of $\even(c_i)$ and $\even(x_i)$ would contain $X_j \cup X_k$ along with the vertex itself, and hence one of $\evendist(c_i)$ and $\evendist(x_i)$ would have size at least $2d_2+1>d_2+1$, which contradicts \eqref{the hypothesis}. This implies that $\dist_G(c_i,c_j)$, $\dist_G(c_i,c_k)$, and $\dist_G(c_j,c_k)$ have pairwise distinct parities, which is a contradiction. We have thus shown that $q\le 2$, which implies that the block graph of $G$ is a tree with at most two leaves and hence a path.
\end{proof}

\begin{lem}\label{lem:auxiliary-endpoint}
Let $G$ be a graph with $n=|V(G)|=2k-1$ for some $k\ge 2$. Suppose there exist $A,B\subseteq V(G)$ such that $A\ne \emptyset$, $B\ne \emptyset$, $A\cap B=\emptyset$, and $ab\notin E(G)$ for all $a\in A$ and $b\in B$. Further suppose that $\deg_G(v)\ge k-1$ for all $v\in A\cup B$ and $\deg_G(v)\ge k$ for all $v\notin A\cup B$. Then $G$ has a Hamiltonian path with one endpoint in $A$ and the other endpoint in $B$.
\end{lem}
\begin{proof}
We start by proving some edge-addition rules which preserve traceability.

Rule 1: Suppose $x,y\in A$ are non-adjacent. Then $G$ has a Hamiltonian path with one endpoint in $A$ and the other endpoint in $B$ if and only if $G+xy$ does.

Proof of Rule 1: Note that $x,y\in A$ implies that
\begin{equation}\label{rule-1-ineq}
\deg_G(x)+\deg_G(y)\ge 2k-2=n-1.
\end{equation}
The forward implication is clear. We now prove the reverse implication. Consider a Hamiltonian path $P\coloneqq v_1\dots v_n$ in $G+xy$ with $v_1\in A$ and $v_n\in B$. If $P$ does not contain $xy$, then we are done. Now suppose $P$ contains $xy$. Then $x=v_i$ and $y=v_{i+1}$ for some $1\le i\le n-2$.

If $yv_1\in E(G)$ then
\[
v_iv_{i-1}\dots v_1v_{i+1}v_{i+2}\dots v_n
\] is the required path in $G$.

If $xv_{j-1},yv_j \in E(G)$ for some $2\le j\le i$, then
\[
v_1v_2\dots v_{j-1}v_iv_{i-1}\dots v_jv_{i+1}v_{i+2}\dots v_n
\] is the required path in $G$.

If $xv_j, yv_{j+1}\in E(G)$ for some $i+2\le j\le n-1$, then
\[
v_1v_2\dots v_i v_jv_{j-1}\dots v_{i+1}v_{j+1}v_{j+2}\dots v_n
\] is the required path in $G$. Hence if $G$ does not contain a Hamiltonian path with one endpoint in $A$ and the other endpoint in $B$, then we must have:
\begin{enumerate}
\item $yv_1\notin E(G)$;
\item $\{xv_{j-1},yv_j\}$ contains at most one edge in $G$ for each $2\le j\le i$; and
\item $\{xv_j, yv_{j+1}\}$ contains at most one edge in $G$ for each $i+2\le j\le n-1$.
\end{enumerate}
Since $y=v_{i+1}$, we have that $yv_{i+2}\in E(G)$. This gives $\le n-2$ edges incident with either $x$ or $y$ in $G$. The incidences that remain to be checked are $xv_i$ and $yv_{i+1}$ (which are self-edges and hence not allowed), $xv_{i+1}=xy$ which is not an edge since $x$ and $y$ were assumed to be non-adjacent in $G$, and $xv_n$ which is not an edge since $x\in A$, $v_n \in B$, and it was assumed that there are no edges between a vertex in $A$ and a vertex in $B$. There are therefore at most $n-2$ edges in $G$ that are incident with either $x$ or $y$, contradicting Inequality \eqref{rule-1-ineq}. Thus $G$ contains a Hamiltonian path with one endpoint in $A$ and the other endpoint in $B$.

Rule 2: Suppose $x,y\in B$ are non-adjacent. Then $G$ has a Hamiltonian path with one endpoint in $A$ and the other endpoint in $B$ if and only if $G+xy$ does.

Proof of Rule 2: The proof of Rule 1 can be repeated with $A$ and $B$ interchanged wherever they occur.

Rule 3: Suppose $x,y\in V(G)\setminus (A\cup B)$ are non-adjacent. Then $G$ has a Hamiltonian path with one endpoint in $A$ and the other endpoint in $B$ if and only if $G+xy$ does.

Proof of Rule 3: Note that $x,y\in V(G)\setminus (A\cup B)$ implies that
\begin{equation}\label{rule-3-ineq}
\deg_G(x)+\deg_G(y)\ge 2k=n+1.
\end{equation}
The forward implication is clear. We now prove the reverse implication. Consider a Hamiltonian path $P\coloneqq v_1\dots v_n$ in $G+xy$ with $v_1\in A$ and $v_n\in B$. If $P$ does not contain $xy$, then we are done. Now suppose $P$ contains $xy$. Then $x=v_i$ and $y=v_{i+1}$ for some $2\le i\le n-2$.

If $xv_{j-1},yv_j \in E(G)$ for some $2\le j\le i$, then
\[
v_1v_2\dots v_{j-1}v_iv_{i-1}\dots v_jv_{i+1}v_{i+2}\dots v_n
\] is the required path in $G$.

If $xv_j, yv_{j+1}\in E(G)$ for some $i+2\le j\le n-1$, then
\[
v_1v_2\dots v_i v_jv_{j-1}\dots v_{i+1}v_{j+1}v_{j+2}\dots v_n
\] is the required path in $G$.

Hence if $G$ does not contain a Hamiltonian path with one endpoint in $A$ and the other endpoint in $B$, then we must have:
\begin{enumerate}
\item $\{xv_{j-1},yv_j\}$ contains at most one edge in $G$ for each $2\le j\le i$; and
\item $\{xv_j, yv_{j+1}\}$ contains at most one edge in $G$ for each $i+2\le j\le n-1$.
\end{enumerate}
Since $y=v_{i+1}$, we have that $yv_{i+2}\in E(G)$. This gives $\le n-2$ edges incident with either $x$ or $y$ in $G$. The incidences that remain to be checked are $xv_i$ and $yv_{i+1}$ (which are self-edges and hence not allowed), $xv_{i+1}=xy$ which is not an edge since $x$ and $y$ were assumed to be non-adjacent in $G$, and $yv_1$ and $xv_n$, which are possibly edges in $G$. There are therefore at most $n$ edges in $G$ that are incident with either $x$ or $y$, contradicting Inequality \eqref{rule-3-ineq}. Thus $G$ contains a Hamiltonian path with one endpoint in $A$ and the other endpoint in $B$.

We now prove the lemma. Let $C\coloneqq V(G)\setminus (A\cup B)$ From rules 1,2, and 3, it follows that we may add any missing edges between two vertices in $A$, two vertices in $B$, or two vertices in $C$ while preserving the traceability of $G$. It therefore suffices to prove the lemma when $G[A], G[B],$ and $G[C]$ are all cliques. Consider $b\in B$. We have
\begin{equation*}
k-1\le \deg_G(b)\le n-|A|-1=2k-2-|A|,
\end{equation*}
where the second inequality follows from the fact that there are no edges between a vertex in $A$ and a vertex in $B$. This gives $|A|\le k-1$. By a similar argument, we have $|B|\le k-1$. Since $A$ is a clique, we have
\begin{equation*}
|N_G(a)\cap C|=\deg_G(a)-(|A|-1)\ge k-1-|A|+1=k-|A|\ge 1.
\end{equation*}
Hence every vertex in $A$ has a neighbour in $C$. Similarly, every vertex in $B$ has a neighbour in $C$. Note that
\begin{equation*}
|C|=2k-1-|A|-|B|\ge 2k-1-2(k-1)=1.
\end{equation*}

If $C$ is a singleton containing a vertex $z$, then $z$ is adjacent to every vertex in $A\cup B$, as seen above. Since $A$ and $B$ are cliques, we may order the vertices of $A$ and $B$ as $v_1,\dots,v_{|A|}$ and $w_1,\dots,w_{|B|}$ respectively to obtain a Hamiltonian path
\[
v_1\dots v_{|A|}zw_1\dots w_{|B|}
\] in $G$ with one endpoint in $A$ and the other endpoint in $B$.

Now suppose $|C|\ge 2$. Consider the sets
\[
S_1\coloneqq \bigcup\limits_{a\in A}(N_G(a)\cap C) \text{ and } S_2\coloneqq \bigcup\limits_{b\in B}(N_G(b)\cap C).
\]

Both these sets are nonempty, as seen above. Suppose there did not exist distinct vertices $x,y\in V(G)$ such that $x\in S_1$ and $y\in S_2$. Then $S_1=S_2=\{z\}$ for some $z\in C$. We then have
\begin{align*}
1 &= |S_1| =|\bigcup\limits_{a\in A}(N_G(a)\cap C)|\\
&\ge |N_G(a_0)\cap C| =\deg_G(a_0)-(|A|-1)\\
&\ge k-1-|A|+1 =k-|A| \\
&\ge 1
\end{align*}
for any $a_0\in A$, which gives $\deg_G(a_0)=|A|$ for any $a_0\in A$. Using $k-1\le \deg_G(a_0)=|A|\le k-1$, we obtain $|A|=k-1$. By a similar argument, we have $|B|=k-1$. Then
\[
2k-1=|V(G)|=|A\cup B\cup C|=|A|+|B|+|C|\ge 2(k-1)+2=2k,
\] which is a contradiction. Hence there exist distinct vertices $x,y\in V(G)$ such that $x\in S_1$ and $y\in S_2$.

Choose $a\in A$ adjacent to $x$ and $b\in B$ adjacent to $y$. Since $A,B,$ and $C$ are cliques, we can order their vertices as $v_1\dots v_{|A|-1}a$, $bw_1\dots w_{|B|-1}$, and $xu_1\dots u_{|C|-2}y$ to obtain the required Hamiltonian path
\[
v_1\dots v_{|A|-1}axu_1\dots u_{|C|-2}ybw_1\dots w_{|B|-1}
\] in $G$.
\end{proof}

\begin{lem}\label{lem:rooted-end-block}
Let $G$ be connected such that some vertex $t\in V(G)$ has degree one with neighbour $c$ and $G-t$ is 2-connected. Let $k\ge 2$, and assume that:
\begin{enumerate}
\item $|V(G)|\le 2k+2$;
\item at most one vertex in $V(G)\setminus\{t,c\}$ has degree less than $k$ in $G$; and
\item $\evendist(v)\le k+1$ for every $v\in V(G)$.
\end{enumerate}
Then $G$ is traceable. Equivalently, $G-t$ has a Hamiltonian path with endpoint $c$.
\end{lem}
\begin{proof}
Since $G-t$ is 2-connected, \cite[Proposition 1.4.2]{Diestel} gives $\delta(G-t)\ge \kappa(G-t)\ge 2$ and hence that $t$ is the only vertex of degree one in $G$ and that $\deg_G(c)\ge 3$.

By (2), we have that at most three vertices have degree less than $k$ in $G$, and hence that $k\le d_4$. Since $G$ is connected, we have $d_1\ge 1$. Since $G-t$ is 2-connected, every vertex in $G$ other than $t$ has degree $\ge 2$, giving $d_2\ge 2$. We therefore have $d_3\ge 2$. Suppose $d_3=2$. Then $d_2=2$. Since $\deg_G(c)\ge 3$, we must have $v_1,v_2\in V(G)\setminus \{t,c\}$ such that $\deg_G(v_1)=\deg_G(v_2)=2$. By (2), there is at most one vertex in $V(G)\setminus \{t,c\}$ with degree less than $k$ in $G$, and hence we must have $k<3$. Since $k\ge 2$, we obtain $k=2$. Since $k=d_2$, (3) gives that $G$ satisfies \eqref{the hypothesis}. Using (1) along with Lemma \ref{lem:nle 2d+1} and Theorem \ref{thm:2d+2} then gives that $G$ is traceable. We may therefore suppose that $d_3\ge 3$.

Suppose $n\le 2k+1$. Then $k\ge \frac{n-1}{2}$. For each $4\le i<\frac{n+1}{2}$, we therefore have
\[
d_i\ge d_4\ge k\ge \frac{n-1}{2}\ge i.
\] We have already shown that $d_1\ge 1$ and $d_2\ge 2$, and assumed that $d_3\ge 3$. Lemma \ref{lem:Chvatal} therefore gives that $G$ is traceable.

Hence only the case $n=2k+2$ remains to be checked. If there is a vertex in $V(G)\setminus\{t,c\}$ with degree less than $k$, call it $z$. Start with $G$ and apply the edge-addition rules from Lemma \ref{lem:ordinary-path-closure} and Lemma \ref{lem:distance-3-closure}, with the restriction that no added edge is incident with $t$ or with $z$ (if it exists). Let $J$ be the terminal graph obtained after adding edges. Then $J$ is traceable if and only if $G$ is traceable. Note that $\deg_J(t)=1$, $\deg_J(z)=\deg_G(z)$ (if it exists), and $J-t$ is 2-connected. Suppose $J$ is not traceable. Let $b_1\le \dots \le b_n$ be the degree sequence of $J$. Since $J-t$ is 2-connected, every vertex in $J$ other than $t$ has degree $\ge 2$, and $\deg_J(c)\ge 3$.

For $4\le i\le k$, we have
\[
i\le k\le d_4\le d_i\le b_i.
\] We also have $b_i\ge d_i\ge i$ for $1\le i\le 3$. We therefore obtain $b_i\ge i$ for $1\le i\le k$. Note that
\[
\frac{n+1}{2}=\frac{2k+3}{2}=k+\frac{3}{2}.
\] Since $J$ is not traceable, Lemma \ref{lem:Chvatal} gives
\[
b_{n-k-1+1}=b_{k+2}<n-k-1=k+1
\] and hence $b_{k+2}\le k$. Therefore, at least $k+2$ vertices of $J$ have degree at most $k$.

Define
\begin{equation*}
S\coloneqq \{v\in V(J)\setminus\{t,c,z\}:\deg_J(v)=k\} \text{ (with $z$ omitted if it does not exist).}
\end{equation*}

Also define
\begin{equation*}
R\coloneqq \{v\in V(J):\deg_J(v)\ge k+1\} \text{ and }
T\coloneqq V(J)\setminus (S\cup R).
\end{equation*}

Then
\begin{equation}\label{T-containment}
\{t,z\}\subseteq T \subseteq \{t,c,z\} \text{ (with $z$ omitted if it does not exist)}
\end{equation}
and
\begin{equation}\label{R-upper-bound}
|R|\le n-(k+2)=k.  
\end{equation}
Since $|T\cup S|\ge k+2\ge 4$ and $|T|\le 3$, we have $S\ne\emptyset$. Since $S\subseteq V(J)\setminus\{t,c,z\}$, we have that $\deg_G(v)\ge k$ for every $v\in S$. We also have $\deg_J(v)=k$ for every $v\in S$, by the definition of $S$. Then $k=\deg_J(v)\ge \deg_G(v)\ge k$ for every $v\in S$, and hence
\begin{equation}\label{eqn:s-degree-preserved-1}
\deg_J(v)=\deg_G(v)=k \text{ for every } v\in S,
\end{equation} implying that no added edge is incident with a vertex in $S$. For any $s\in S$ and $r\in R$, we have 
\begin{equation*}
\deg_J(s)+\deg_J(r)\ge k+(k+1)=n-1.   
\end{equation*}
Hence if $s$ and $r$ were non-adjacent in $J$, then we would be able to apply the edge-addition rule of Lemma \ref{lem:ordinary-path-closure}, contradicting the terminality of $J$ (Note that $t,z\notin S\cup R$). Hence $sr\in E(J)$ for all $s\in S$ and all $r\in R$. Since no added edge is incident with a vertex in $S$, we must have $sr\in E(G)$ for all $s\in S$ and all $r\in R$. 

Suppose there exist distinct non-adjacent vertices $r$ and $r'$ in $R$. Then $\deg_J(r)+\deg_J(r')\ge 2k+2>n-1$, implying that we would be able to apply the edge-addition rule of Lemma \ref{lem:ordinary-path-closure}, which contradicts the terminality of $J$. Hence $R$ induces a clique in $J$. Consider $x\in S$ and define
\begin{equation}\label{eqn:mxdefn}
m_x\coloneqq |N_G(x)\cap T|.  
\end{equation}
From Equation \eqref{eqn:s-degree-preserved-1}, we have $\deg_J(x)=\deg_G(x)=k$ and hence $N_G(x)=N_J(x)$. We also know that $x$ has $|R|$ neighbours in $R$ and hence $k-|R|-m_x$ neighbours in $S$. Since $|S|=2k+2-|R|-|T|$, $x$ has
\begin{equation*}
2k+1-|R|-|T|-k+|R|+m_x=k+1-|T|+m_x
\end{equation*}
non-neighbours in $S$ (other than itself). Consider some such non-neighbour $y\in S$ of $x$. From Equation \eqref{eqn:s-degree-preserved-1}, we have $\deg_J(x)+\deg_J(y)=2k=n-2$. Suppose $\dist_J(x,y)\ge 3$. Then we would be able to apply the edge-addition rule of Lemma \ref{lem:distance-3-closure}, contradicting the terminality of $J$. Hence $\dist_J(x,y)\le 2$. Since $y$ is a non-neighbour of $x$, we have $\dist_J(x,y)= 2$. Since $x,y\in S$, no added edge is incident with either of them, and hence any path $xvy$ of length 2 in $J$ must also be in $G$. This implies that $\dist_G(x,y)= 2$. We therefore obtain that all non-neighbours of $x$ in $S$ are at an even distance from $x$. Since $N_G(t)=\{c\}$, we have $\dist_G(x,t)=\dist_G(x,c)+1$, and hence that exactly one of $c$ and $t$ is at an even distance from $x$. Using the above and hypothesis (3), we obtain
\begin{equation*}
 1+k+1-|T|+m_x+1\le k+1,   
\end{equation*} which simplifies to
\begin{equation}\label{T-lower-bound}
|T|\ge m_x+2.
\end{equation}
In particular, $|T|\ge 2$.

Suppose $|T|=2$. Then Inequality \eqref{T-lower-bound} gives us 
\begin{equation}\label{eqn:mxzeroforx}
m_x=0 \text{ for all } x\in S
\end{equation}
and \eqref{T-containment} gives us $T=\{t,c\}$ or $T=\{t,z\}$.

First consider the case $T=\{t,c\}$. From equations \eqref{eqn:mxdefn} and \eqref{eqn:mxzeroforx}, we have that $c$ is not adjacent to any vertices in $S$. Since $\deg_G(c)\ge 3$, there exists $r_1\in R$ such that $cr_1\in E(G)$. Since $sr_1\in E(G)$ for every $s\in S$, we have that $\dist_G(c,s)=2$ for every $s\in S$. By hypothesis (3), we therefore have $1+|S|\le 1+k$, which gives $|S|\le k$. We already have $|R|\le k$ (Inequality \eqref{R-upper-bound}) and
\begin{equation*}
|S|+|R|=n-|T|=2k+2-2=2k,  
\end{equation*}
and hence $|S|=|R|=k$. Writing $R=\{r_1,\dots,r_k\}$ and $S=\{s_1,\dots,s_k\}$, we obtain the Hamiltonian path 
\[
tcr_1s_1r_2s_2\dots r_ks_k
\] in $J$, which contradicts the non-traceability of $J$.

Now consider the case $T=\{t,z\}$. Then $c\in R$. From equations \eqref{eqn:mxdefn} and \eqref{eqn:mxzeroforx}, we have that $z$ is not adjacent to any vertices in $S$. Since $\deg_G(z)\ge 2$, $z$ has a neighbour $r_2\in R$. Each $x\in S$ has
\begin{equation*}
k+1-|T|+m_x=k+1-2=k-1
\end{equation*} non-neighbours in $S$ at a positive even distance from it. Since $x\in S$ and $c\in R$, we have $xc\in E(G)$. The paths $xct$ and $xr_2z$ show that $t$ and $z$ are at an even distance (of two) from $x$. Finally, $x$ is at an even distance from itself. We have therefore obtained $k-1+3=k+2>k+1$ vertices at an even distance from $x$ in $G$, contradicting hypothesis (3).

Since $|T|\le 3$ (using \eqref{T-containment}), the only case that remains to be checked is $|T|=3$. In this case, $T=\{t,c,z\}$ (using \eqref{T-containment}) and $m_x\le 1$ for each $x\in S$ (using Inequality \eqref{T-lower-bound}). Hence no vertex in $S$ is adjacent to both $c$ and $z$. Define
\begin{equation*}
K\coloneqq J-\{t,c,z\},\text{ } A\coloneqq N_J(c)\cap V(K),\text{ and } B\coloneqq N_J(z)\cap V(K).
\end{equation*}
Note that $A\cap B\cap S=\emptyset$. Since $\deg_{J-t}(c)\ge 2$ and $\deg_{J-t}(z)\ge 2$, both $c$ and $z$ must have a neighbour (in $J$) among the vertices in $V(K)$. Hence $A$ and $B$ are nonempty. Since no added edges were incident to $z$, we have $B=N_G(z)\cap V(K)$. We have $V(K)=V(J)\setminus T=S\cup R$ and $|V(K)|=2k+2-3=2k-1$. Since $R$ induces a clique in $J$ and every vertex of $R$ is adjacent to every vertex of $S$ in $J$, we have that every vertex of $R$ is adjacent to every vertex of $S\cup R$ in $J$, and hence adjacent to every vertex in $K$. Consider some $x\in S$. We have $\deg_J(x)=k$. Since the only neighbour of $t$ in $J$ is $c\notin S$, we have that $x$ is not adjacent to $t$ in $J$. If $x\in A\cup B$, then $x$ is adjacent to $c$ or $z$ in $J$. Since $x$ cannot be adjacent to both $c$ and $z$ in $J$, we must have that $x$ is adjacent to exactly one of $c$ and $z$ in $J$, and hence that $\deg_K(x)=\deg_J(x)-1=k-1$. If $x\notin A\cup B$, then $x$ is adjacent to neither $c$ nor $z$ in $J$, and hence $\deg_K(x)=\deg_J(x)=k$.

We therefore have
\begin{equation}\label{eqn:degrees of x in s}
\deg_K(x)=
\begin{cases}
k-1 & \text{if } x\in S\cap (A\cup B)\\
k & \text{if } x\in S\setminus (A\cup B)
\end{cases}
\end{equation}

We now consider two cases, based on whether $A\cap S$ is empty or not.

Suppose $A\cap S\ne\emptyset$. Then write
\begin{equation*}
A_0\coloneqq A\cap S.
\end{equation*}
Choose $a\in A_0$. We have $c\in N_J(a)=N_G(a)$ (since $a\in S$), and hence
\begin{equation*}
1\ge m_a= |N_G(a)\cap T|\ge |\{c\}|=1,
\end{equation*}
giving $m_a=1$. Then $a$ has
\begin{equation*}
k+1-|T|+m_a=k+1-3+1=k-1
\end{equation*}
non-neighbours in $S$, all at a distance of two from it. The path $act$ shows that $t$ is also at a distance of two from $a$. Since $a$ therefore has at least $k+1$ vertices at an even distance from it, $z$ cannot be at an even distance from it (due to hypothesis (3)), and in particular, cannot be at a distance of two from it. Note that $z\notin N_G(a)$, since $a\in S$ and $c\in N_G(a)$. Consider some $b\in B\cap R$. We have $ab\in E(G)$ since $a\in S$ and $b\in R$, and $bz\in E(G)$, since $B\subseteq N_J(z)=N_G(z)$. Hence $abz$ is a path of length two in $G$, giving us $\dist_G(a,z)=2$, which is a contradiction. Therefore $B\cap R=\emptyset$. Hence $B\subseteq S$. Suppose there exists $b\in B$ such that $ab\in E(J)$. Then $ab\in E(G)$, since no added edge is incident with $a\in S$. This would again give a path $abz$ in $G$, showing that $\dist_G(a,z)=2$, which is a contradiction. Hence there are no edges between a vertex in $A_0$ and a vertex in $B$. Observe that
\begin{enumerate}
\item $|V(K)|=2k-1$;
\item $A_0\ne \emptyset$;
\item $B\ne\emptyset$;
\item $A_0\cap B=(A\cap S)\cap B=A\cap B\cap S=\emptyset$;
\item $ab\notin E(G)$ for all $a\in A_0$ and $b\in B$;
\item $\deg_K(v)\ge k-1$ for $v\in S\cap (A\cup B)=(S\cap A)\cup (S\cap B)=A_0\cup B$; and
\item $\deg_K(v)\ge k$ for $v\in (S\cap (V(K)\setminus (A\cup B)))\cup R = V(K)\setminus (A_0\cup B)$.
\end{enumerate}
Thus, by Lemma \ref{lem:auxiliary-endpoint}, $K$ has a Hamiltonian path
\[
P\coloneqq a'v_1\dots v_{2k-3}b',
\] where $a'\in A_0$ and $b'\in B$. Then
\[
tcPz=tca'v_1\dots v_{2k-3}b'z
\] is a Hamiltonian path in $J$, which contradicts the non-traceability of $J$.

Now consider the case where $A\cap S=\emptyset$. Then $A\subseteq R$. Since $G-t$ is 2-connected, $c$ has a neighbour $r_0\in V(K)$ (in $G$). Then
\begin{equation*}
r_0\in N_G(c)\cap V(K)\subseteq N_J(c)\cap V(K)=A\subseteq R.
\end{equation*}
Since every vertex in $S$ is adjacent to $r_0$ in $G$, we have a path $cr_0s$ of length two between $c$ and any $s\in S$. Also note that 
\begin{equation*}
N_G(c)\cap S=(N_G(c)\cap S)\cap S\subseteq (N_J(c)\cap V(K))\cap S=A\cap S=\emptyset,
\end{equation*}
and hence no $s\in S$ is a neighbour of $c$ in $G$. Therefore, all $s\in S$ satisfy $\dist_G(c,s)=2$, and thus hypothesis (3) forces $|S|\le k$. Since $|S|+|R|=2k-1$ and $|R|\le k$, we have $|R|\in\{k-1,k\}$.

Consider the case where $|R|=k$. Then $|S|=k-1$. For any $x\in S$, we have that $\deg_J(x)=\deg_G(x)=k$ (Equation \eqref{eqn:s-degree-preserved-1}) and that $xr\in E(G)$ for all $r\in R$. Thus $x$ has no neighbours outside $R$. From Equation \eqref{eqn:degrees of x in s}, we have $\deg_G(v)=\deg_J(v)=k-1$ for $v\in S\cap (A\cup B)$ and hence $S\cap (A\cup B)=\emptyset$, which yields $B\cap S=\emptyset$. We therefore have $B\subseteq R$. Suppose $A=B=\{v_0\}$ for some $v_0\in R$. Then deleting $v_0$ would separate $c$ from $V(K)\setminus \{v_0\}$, thus contradicting the 2-connectivity of $J-t$. Hence there exist $a\in A$ and $b\in B$ such that $a\ne b$. Recall that every vertex in $R$ is adjacent to every vertex in $K$. Order the vertices of $R$ as $r_1,\dots,r_k$ with $r_1\coloneqq a$ and $r_k\coloneqq b$, and the vertices of $S$ as $s_1,\dots,s_{k-1}$. Then
\[
r_1s_1\dots r_{k-1}s_{k-1}r_k
\] is a Hamiltonian path in $K$. Since $a$ and $b$ are neighbours of $c$ and $z$ respectively in $J$, we have a Hamiltonian path
\[
tcr_1s_1\dots r_{k-1}s_{k-1}r_kz
\] in $J$, thus contradicting the non-traceability of $J$.

Hence it only remains to consider the case where $|R|=k-1$. In this case, $|S|=k$. Let $B_0\coloneqq B\cap S$. Consider the induced subgraph $J[S]$ and $x\in B_0$. We have that $\deg_J(x)=k$ and that $x$ is adjacent to all $k-1$ vertices of $R$ in $J$, along with being adjacent to $z$ in $J$. Hence $x$ does not have any neighbours among the vertices of $S$ in $J$, implying that $\deg_{J[S]}(x)=0$. Now consider $y\in S\setminus B_0$. We have $\deg_J(y)=k$, $zy\notin E(J)$, $c\notin S$ and hence $y\ne c$, implying that $ty\notin E(J)$, $A\cap S=\emptyset$ implying that $cy\notin E(J)$, and that $y$ is adjacent to all $k-1$ vertices of $R$ in $J$. Hence $y$ has exactly one neighbour in $J$ among the vertices of $S$, implying that $\deg_{J[S]}(y)=1$. Thus $J[S]$ consists of $|B_0|$ isolated vertices (the vertices in $B_0$) and $\frac{|S|-|B_0|}{2}$ components of size two (the vertices in $S\setminus B_0$). We have
\begin{equation}\label{size-bounds}
|B_0|\le |B|\le |N_J(z)|=|N_G(z)|=\deg_D(z)<k.
\end{equation}

Suppose that $|B_0|>0$. The number of components of $J[S]$ is
\begin{equation*}
M\coloneqq |B_0|+\frac{k-|B_0|}{2}=\frac{k+|B_0|}{2}\le \frac{2k-1}{2}=k-\frac{1}{2},
\end{equation*}
where the inequality follows from the fact that $|B_0|<k$ (Inequality \eqref{size-bounds}), and hence
\begin{equation*}
M\le k-1=|R|, \text{ since $M$ is an integer}.
\end{equation*}
Fix some $b_0\in B_0$ and order the components of $J[S]$ such that the singleton component $\{b_0\}$ is last. Recall that every vertex in $R$ is adjacent to every other vertex in $K$. We partition $R$ into $M$ sets and order these sets arbitrarily. We construct a Hamiltonian path $P$ in $K$ by starting with some $a\in A$ and alternating between attaching to it one set of vertices from the partition of $R$ and one component of $J[S]$. Then $tcPz$ is a Hamiltonian path in $J$, contradicting the non-traceability of $J$.

Now suppose that $|B_0|=0$. Then $B\subseteq R$ and all components of $J[S]$ are of size two. In particular, $|S|=k$ is even. Since $2\le \deg_G(z)<k$, we have that $k\ge 4$. The graph $J[S]$ has
\begin{equation*}
\frac{|S|}{2}=\frac{k}{2}    
\end{equation*} components, and since $k\ge 4$, we have
\begin{equation*}
\frac{k}{2}+1\le k-1=|R|.   
\end{equation*}
Suppose $A=B=\{v_0\}$ for some $v_0\in R$. Then deleting $v_0$ would separate $c$ from $V(K)\setminus \{v_0\}$, thus contradicting the 2-connectivity of $J-t$. Hence there exist $a\in A$ and $b\in B$ such that $a\ne b$. Fix such $a$ and $b$. We now construct a Hamiltonian path $P$ in $K$. First partition $R\setminus \{a\}$ into $\frac{k}{2}$ sets and order them such that the set containing $\{b\}$ is last. Start with $a$ and alternate between attaching to it one set of vertices from the partition of $R\setminus \{a\}$ and one component of $J[S]$. While attaching the last set of vertices from the partition of $R$, order the vertices within that set such that $b$ is last. Now observe that $tcPz$ is a Hamiltonian path in $J$, thus contradicting the non-traceability of $J$ and concluding the proof.
\end{proof}

\begin{thm}
Suppose $G$ satisfies \eqref{the hypothesis}. If $G$ is not traceable, then $G$ is 2-connected.
\end{thm}
\begin{proof}
Suppose $G$ satisfies \eqref{the hypothesis} and is neither traceable nor 2-connected. Then $G$ has a cutvertex. By Theorem \ref{thm:block-cutvertex path}, the block graph of $G$ is a path. Let $B_1$ and $B_2$ be its end-blocks, with cutvertices $c_1$ and $c_2$.
\begin{equation*}
X\coloneqq V(B_1)\setminus\{c_1\},\text{ } Y\coloneqq V(B_2)\setminus\{c_2\},\text{ }, p\coloneqq |X|,\text{ and } q\coloneqq |Y|.
\end{equation*}

Since $G$ satisfies \eqref{the hypothesis} and is not traceable, we may use Inequality \eqref{eqn:block-size} to obtain that $p,q\ge d_2$. Define
\begin{equation*}
a\coloneqq |\{x\in X:\dist_G(x,c_1)\text{ is even}\}| \text{ and } b\coloneqq |\{y\in Y:\dist_G(y,c_2)\text{ is even}\}|.
\end{equation*}

Suppose that the block graph of $G$ consists of at least three blocks, and let
\begin{equation*}
Z\coloneqq V(G)\setminus (X\cup Y).
\end{equation*}
Then $z\coloneqq |Z|\ge 2$, since $c_1,c_2\in Z$. Consider some $x\in X\cap N_G(c_1)$. For any $v\in Z\setminus \{c_1\}$, all paths from $v$ to $x$ must contain $c_1$, giving that $v$ is at an even distance from either $x$ or $c_1$, and $c_1\in Z$ is also in $\even(c_1)$. Similarly, for any $y\in Y$, all paths from $y$ to $x$ must contain $c_1$, giving that $y$ is at an even distance from either $x$ or $c_1$. There are $a$ vertices in $X$ that are at an even distance from $c_1$, and $x$ is at an even distance from itself. From this, we obtain
\begin{equation*}
z+q+a+1\le \evendist(c_1)+\evendist(x)\le 2d_2+2,
\end{equation*}
where the second inequality follows from \eqref{the hypothesis}. The exact same process can be repeated for $Y$, finally giving
\begin{equation}\label{eqn:tightened-block-inequality}
q+z+a\le 2d_2+1 \text{ and } p+z+b\le 2d_2+1.
\end{equation}

Suppose there exists $u\in Z\setminus\{c_1,c_2\}$. We have $N_G(u)\subseteq Z$. Suppose $\deg_G(u)\ge d_2$. Then $z\ge d_2+1$. Using this along with $p,q\ge d_2$ in Inequality \eqref{eqn:tightened-block-inequality} gives
\begin{equation*}
z=d+1,\text{ }, p=q=d_2,\text{ and } a=b=0.
\end{equation*}
By the argument used in Theorem \ref{thm:block-cutvertex path}, we obtain that $B_1$ and $B_2$ are both copies of $K_{d_2+1}$ and that $G$ has no vertex of degree less than $d_2$. Since $\deg_G(u)\ge d_2$ and $z=d_2+1$, $N_G(u)\subseteq Z$ gives $\deg_G(u)=d_2$ and hence $N_G(u)= Z$. In particular, we obtain $\{c_1,c_2\}\subseteq N_G(u)$. We therefore have that all $2d_2$ vertices in $X\cup Y$ are at an even distance (of two) from $u$. This fact, along with $u\in\even(u)$, gives $\evendist(u)\ge 2d_2+1>d_2+1$, contradicting \eqref{the hypothesis}. Hence any $u\in Z\setminus\{c_1,c_2\}$ must have degree less than $d_2$, implying that there is at most one vertex in $Z\setminus\{c_1,c_2\}$. The block graph of $G$ being a path therefore implies that there are exactly three blocks in it, namely the end-blocks $B_1$ and $B_2$, and the middle block consisting of $c_1,c_2$, and possibly a vertex $u$. Suppose that such a $u$ exists. Note that $u$ must be adjacent to both $c_1$ and $c_2$. Consider $\even(c_1)$ and $\even(c_2)$ in this case. There are $a$ vertices in $X$ at an even distance from $c_1$. Since $c_1$ and $c_2$ are adjacent, the vertices in $Y$ at an even distance from $c_1$ are exactly those that are at an odd distance from $c_2$. There are $q-b$ such vertices. Finally, $c_1$ is at an even distance from itself. We therefore obtain
\begin{equation*}
a+q-b+1\le \evendist(c_1)\le d_2+1,
\end{equation*}
and repeating the argument for $c_2$ gives
\begin{equation*}
b+p-a+1\le \evendist(c_2)\le d_2+1.
\end{equation*}
Adding these two inequalities gives $p+q\le 2d_2$. Along with $p,q\ge d_2$, this gives $p=q=d_2$. Since the only vertex in $G$ with degree less than $d_2$ is $u$, we have that $B_1$ and $B_2$ are copies of $K_{d_2+1}$, which gives $a=b=0$. This in turn implies that $\{u\}\cup X\cup Y\subseteq \even(u)$, and hence that $\evendist(u)\ge 2d_2+1>d_2+1$, contradicting \eqref{the hypothesis}. The only remaining case with at least three blocks is therefore one where $Z=\{c_1,c_2\}$. In this case, $c_1c_2$ is a bridge in $G$, contradicting Lemma \ref{lem:nobridge}.

Now suppose that there are exactly two blocks, $B_1$ and $B_2$. They must then be joined at a single vertex $c_1=c_2=c$. Consider $x\in X\cap N_G(c)$. For any $y\in Y$, all paths from $y$ to $x$ must contain $c$, giving that $y$ is at an even distance from either $x$ or $c$. There are $a$ vertices in $X$ at an even distance from $c$, and $c$ and $x$ are each at an even distance from themselves. From this, we obtain
\begin{equation*}
q+a+2\le \evendist(c)+\evendist(x)\le 2d_2+2
\end{equation*}
and hence $q+a\le 2d_2$. Repeating this argument for some $y\in Y\cap N_G(c)$ yields $p+b\le 2d_2$. We therefore have $p,q\le 2d_2$. Fix some $y\in Y\cap N_G(c)$ and consider
\begin{equation*}
D_X\coloneqq G[X\cup \{c,y\}].
\end{equation*}
Observe that
\begin{equation*}
\dist_{D_X}(v_1,v_2)=\dist_G(v_1,v_2) \text{ for all } v_1,v_2\in V(D_X)=X\cup \{c,y\},
\end{equation*}
since any edge between two vertices in $X \cup \{c\}$ are preserved and all paths between $y$ and a vertex in $X$ contain $c$ (which is adjacent to $y$). We therefore also have
\begin{align*}
\even_{D_X}(v) &=\{u\in V(D_X):\dist_{D_X}(u,v)\text{ is even}\}\\
&=\{u\in V(D_X):\dist_{G}(u,v)\text{ is even}\}\subseteq \{u\in V(G):\dist_{G}(u,v)\text{ is even}\}\\
&=\even_G(v),
\end{align*}
and hence
\begin{equation*}
\evendist^{D_X}(v)\le \evendist^G(v) \text{ for all } v\in V(D_X).
\end{equation*}
Now consider the statement of Lemma \ref{lem:rooted-end-block}. We have that $D_X$ is connected, that $y\in V(D_X)$ has degree one in $D_X$ (since $N_{D_X}(y)=\{c\}$), and that $D_X-y$ is 2-connected, since it is a block in $G$. This gives $\delta(D_x-y)\ge \kappa(D_X-y)\ge 2$ and hence that the second-smallest degree in $D_X$ is at least $2$, from which it follows that $d_2\ge 2$. Let $k=d_2$ as in the statement of Lemma \ref{lem:rooted-end-block}. Then $|V(D_X)|=2d_2+2$. Now suppose there were two vertices $v_1,v_2\in V(D_X)\setminus \{c,y\}=X$ with degree less than $d_2$ in $D_X$. We have $N_G(v)\subseteq X\cup \{c\}$ for all $v\in X$. Since $D_X$ is an induced subgraph of $G$, we therefore have $\deg_{D_X}(v)=\deg_G(v)$ for all $v\in X$, thus giving $\deg_G(v_1)<d_2$ and $\deg_G(v_2)< d_2$, which is a contradiction. Hence there is at most one vertex in $V(D_X)\setminus \{c,y\}$ with degree less than $d_2$ in $D_X$. As seen above, we have $\evendist^{D_X}(v)\le \evendist^G(v)$ for every $v\in V(D_X)$. Using \eqref{the hypothesis} gives $\evendist^{D_X}(v)\le d_2+1$ for every $v\in V(D_X)$. Thus $D_X$ satisfies the hypotheses of Lemma \ref{lem:rooted-end-block} and is therefore traceable, with $D_X-y=B_1$ having a Hamiltonian path with endpoint $c$. By the same argument, one may construct $D_Y$ to show that $B_2$ has a Hamiltonian path with endpoint $c$. Concatenating these two paths at $c$ gives a Hamiltonian path in $G$, thus showing that $G$ is traceable, which is a contradiction, thereby concluding the proof.

This proves (4) from Theorem \ref{thm:main}.
\end{proof}







\printbibliography
\end{document}